\documentclass[12pt]{article}

\usepackage[all]{xy}
\usepackage{amsmath,amsfonts,xcolor,amssymb,amsthm}
\usepackage{hyperref}

\newcommand{\dom}[1]{\mathrm{dom}(#1)}

\newcommand{\sym}[1]{\mathrm{sym}(#1)}

\usepackage{cite}

\theoremstyle{definition}
\newtheorem{lemma}{Lemma}
\newtheorem{theorem}[lemma]{Theorem}

\newtheorem{corollary}[lemma]{Corollary}
\newtheorem{proposition}[lemma]{Proposition}
\newtheorem{definition}[lemma]{Definition}
\newtheorem{ques}[lemma]{Question}

\newcommand{\AC}{\mathsf{AC}}
\newcommand{\ZFC}{\mathsf{ZFC}}
\newcommand{\ZFCM}{\mathsf{T}}
\newcommand{\ZF}{\mathsf{ZF}}
\newcommand{\R}{\mathbb{R}}
\newcommand{\stem}[1]{\mathrm{stem}(#1)}

\newcommand{\HC}{\mathbf{HC}}

\usepackage{comment}

\newcommand{\cz}{\mathbb{C}}
\newcommand{\pz}{\mathbb{P}}
\newcommand{\qz}{\mathbb{Q}}
\newcommand{\sz}{\mathbb{S}}

\begin{document}

\title{Descriptive Choice Principles and How to
Separate Them}
\author{Lucas Wansner\thanks{This work forms part of the doctoral dissertations of the two authors 
\cite{wansner2023aspects, wontner}. The authors would like to thank their supervisor, Benedikt L\"owe, for his help throughout.} \thanks{
Fachbereich Mathematik, University of Hamburg, Hamburg, Germany (\texttt{lucas.wansner@uni-hamburg.de})} \and Ned J.\ H.\ Wontner\footnotemark[1] \thanks{Institute for Logic, Language and Computation, University of Amsterdam, Amsterdam, the Netherlands (\texttt{e.j.h.wontner@uva.nl})}}

\maketitle

\begin{flushright}
\today\end{flushright}

\begin{quote}
\textbf{Abstract.} The axiom of countable choice for reals is one of the most basic fragments of the axiom of choice needed in many parts of mathematics. Descriptive choice principles are a further stratification of this fragment by the descriptive complexity of the sets. In this paper, we provide a separation technique for descriptive choice principles based on Jensen forcing. Our results generalise a theorem by Kanovei.
\end{quote}

\section{Introduction}
The axiom of choice $\mathsf{AC}$ states that every family of non-empty sets has a choice function. 
It can be stratified into fragments based on the index set of the family and the domain the non-empty sets come from. If we write $\mathsf{AC}_X(Y)$ for ``every family indexed by elements of $X$ of non-empty subsets of $Y$ has a choice function'', then $\mathsf{AC}$ is equivalent to ``$\mathsf{AC}_X(Y)$ holds for all sets $X$ and $Y$''.

Among these fragments of $\AC$, $\mathsf{AC}_\omega(\mathbb{R})$ or \emph{the countable axiom of choice for reals}, has a special position. Most of the uses of $\AC$ in ordinary mathematics are required because we want theorems to hold in general; when specialising the results to concrete mathematical objects, they can often be proved in $\mathsf{ZF}$. However, the uses of $\AC$ in the basic foundations of analysis are needed even when working with concrete mathematical objects such as the reals or the complex numbers. The fragment of $\AC$ most fundamental in this respect is the countable axiom of choice for reals $\mathsf{AC}_\omega(\mathbb{R})$. Like $\AC$ itself, $\mathsf{AC}_\omega(\mathbb{R})$ does not follow from $\mathsf{ZF}$ and there are pathological models of $\mathsf{ZF}+\neg\mathsf{AC}_\omega(\mathbb{R})$ where bizarre things can happen: the reals can be a countable union of countable sets or all subsets of the reals can be Borel (\cite{feferman1963independence}, or \cite[pg. 146]{cohen2008set}).\footnote{The first situation implies the second, as singletons are closed, hence ${\mathbf{\Sigma}^0_4}= \mathcal{P}(\mathbb{R})$. These both occur in the Feferman-L\'evy model. This collapse is optimal in $\ZF$ by \cite[Theorem 2.1]{miller2008long}.}

In this paper, we shall look at fragments of this particular fragment called \emph{descriptive choice principles}. These restrict $\mathsf{AC}_\omega(\mathbb{R})$ to those sequences of sets of reals that have a particular description. Here, we are mainly interested in the descriptions given by the projective pointclasses. Choice principles like this were studied in \cite{kanovei1979descriptive} who proved a first separation theorem for some of these principles. In this paper, we improve on Kanovei's result with a method based on Jensen forcing.

In \S\,2, we give all necessary definitions, formally state our main theorem, and derive its consequences.
In \S\,3, we introduce the notion of an \emph{$n$-slicing forcing notion} and prove our main theorem from the assumption that there is a slicing forcing notion.
In \S\,4, we generalise Jensen forcing to construct $n$-slicing forcing notions, thus completing the proof of our main theorem.
Finally, in \S\,5, we list some open questions.

\section{Definitions and the main results}

Our basic theory is Zermelo-Fraenkel set theory $\mathsf{ZF}$ without any fragment of $\AC$.
In some proofs, we need fragments of set theory without the power set axiom. We write 
$\mathsf{ZFC}^-$ for $\mathsf{ZFC}$ without the power set axiom and the collection scheme instead of the replacement scheme.\footnote{Cf.\ \cite{GHJZFCwopowerset} for a discussion of this theory.} The standard model of $\mathsf{ZFC}^-$ is the set of hereditarily countable sets $\HC$.

\subsection{Descriptive choice principles}

As usual in descriptive set theory, we call the elements of Baire space ${\omega^\omega}$ \emph{real numbers} or \emph{reals}. \emph{Par abus de langage}, we write $\mathbb{R}$ to refer to Baire space even though it is not homeomorphic to what is usually called the \emph{real line}. The standard real line plays no role in this paper and therefore, there is no possibility of confusion.
We use the standard notational convention that if $n\in\omega$ and $x\in \mathbb{R}$, then 
$$nx(k) := \left\{
\begin{array}{cl}
n & \mbox{if $k = 0$ and}\\
x(\ell) & \mbox{if $k = \ell+1$;}
\end{array}\right.$$
similarly for finite sequences $s\in\omega^{<\omega}$.

The advantage of working with Baire space is that it is zero-dimensional and therefore, we can easily identify 
subsets of $\mathbb{R}^n$ and even $\mathbb{R}^\omega$ with elements of $\mathbb{R}$ via the Cantor pairing function $\ulcorner.,.\urcorner:\omega\times\omega\to\omega$, e.g., by encoding a sequence $a := \langle a_n:n\in\omega \rangle$ by 
$\widehat{a}(\ulcorner k,\ell\urcorner) := a_k(\ell)$.
Similarly, if $A = \{A_n:n\in\omega\}$ is a family of sets of reals indexed by natural numbers, then $\widehat{A} :=
\{nx:x\in A_n\}$ is a set of reals that contains exactly the same information.
If $A\subseteq\mathbb{R}^2$, we can form its \emph{projection} $\mathsf{proj}(A) := \{x\in\mathbb{R}:\exists y(
(x,y)\in A)\}$.

Since we are working in models without $\AC$, we have to be very specific about our definitions.
As mentioned, in pathological models of $\mathsf{ZF}$, very basic features of descriptive set theory can fail to hold
(e.g., that there are sets which are not Borel). The theory of Borel sets without $\AC$
is complicated: different standard definitions of the Borel hierarchy that are equivalent in $\mathsf{ZF}+\mathsf{AC}_\omega(\mathbb{R})$ fail to be equivalent in general; also, the length of the Borel hierarchy is not determined by $\mathsf{ZF}$ \cite[Theorem 1.2 \& 1.3]{miller2008long}. As a consequence, we
focus on projective descriptive classes in this paper.

A subset $T\subseteq\omega^{<\omega}$ is called a \emph{tree} if it is closed under taking initial segments; if $r$ is a real, a tree is called \emph{computable in $r$} if there is a Turing machine that computes its characteristic function from the oracle $r$; it is called \emph{computable} if it is computable without an oracle. If $T$ is a tree, we define the \emph{set of branches of $T$} to be $[T] := \{x\in\R: \forall n$,  $x{\upharpoonright}n\in T\}$. A set of reals is $\Pi^0_1(r)$ if it is the set of branches of a tree computable in $r$. We define the \emph{lightface projective hierarchy} by recursion: we define $\Pi^1_0(r) := \Pi^0_1(r)$ and say that
a set is $\Sigma^1_{n+1}(r)$ if it is the projection of a $\Pi^1_n(r)$ set and it is 
$\Pi^1_n(r)$ if it is the complement of a $\Sigma^1_n(r)$ set. If $r$ is a computable real, we write $\Sigma^1_n$ and $\Pi^1_n$ for $\Sigma^1_n(r)$ and $\Pi^1_n(r)$, respectively.
The \emph{boldface projective hierarchy} are defined as usual by
$\boldsymbol{\Sigma}^1_n := \bigcup_{r\in\mathbb{R}} \Sigma^1_n(r)$
and
$\boldsymbol{\Pi}^1_n := \bigcup_{r\in\mathbb{R}} \Pi^1_n(r)$; we say that sets in $\boldsymbol{\Sigma}^1_1$ are \emph{analytic}.
We furthermore let $\Delta^1_n(r) := \Sigma^1_n(r)\cap \Pi^1_n(r)$ and 
$\boldsymbol{\Delta}^1_n := \boldsymbol{\Sigma}^1_n\cap\boldsymbol{\Pi}^1_n$. A set is called \emph{projective} if it is in one of the boldface projective classes; we denote the set of projective sets by $\mathbf{Proj}$.

A countable set of reals $C = \{a_n:n\in\omega\}$ is easily seen to be $\Sigma^1_1(\widehat{a})$ by considering the tree
$T := \{(s,nt):t\subseteq a_n\mbox{ and $s$ and $t$ are compatible}\}$ and observing that $C = \mathrm{p}[T]$. 
Not all countable sets are lightface $\Sigma^1_1$.
If we denote the set of countable sets of reals by $\mathbf{Ctbl}$, we therefore obtain the following chain of inclusions:
$$\mathbf{Ctbl} \subseteq \boldsymbol{\Sigma}^1_1 \subseteq ...
\boldsymbol{\Sigma}^1_n \subseteq 
\boldsymbol{\Sigma}^1_{n+1} \subseteq ...
\subseteq\mathbf{Proj}.$$
If we define $\mathbf{F}_\sigma$ as the collection of countable unions of closed sets, then 
$\mathbf{Ctbl}\subseteq \mathbf{F}_\sigma$. It is known that in $\mathsf{ZF}$,
$\mathbf{F}_\sigma \subseteq \boldsymbol{\Sigma}^1_1$ \cite[562D \& 562F]{fremlin2000measure}.

If $\Gamma$ is any of the descriptive classes defined above, we say that a 
countable family $A = \{A_n:n\in\omega\}$ of sets of reals is in $\Gamma$ if each of the elements of the family is in $\Gamma$. We say that it is \emph{uniformly in $\Gamma$} if it is in $\Gamma$ and furthermore $\widehat{A}$ is also in $\Gamma$. We write $\mathrm{unif}\Gamma$ for the collection of countable families of sets of reals that are uniformly in $\Gamma$.

We can now stratify the axiom $\mathsf{AC}_\omega(\mathbb{R})$ according to the descriptive classes that we have considered: let $\Xi$ be any definable class of countable collections of sets of reals, e.g., one of our descriptive classes $\Gamma$ or one of the uniform versions $\mathrm{unif}\Gamma$. Then we write
$$\mathrm{AC}_\omega(\mathbb{R};\Xi)$$
for the statement ``every countable collection of non-empty sets of reals in $\Xi$ has a choice function''. Note that
$\mathsf{AC}_\omega(\mathbb{R};\mathbf{Ctbl})$ is the same principle that is sometimes known as
$\mathsf{CAC}_\omega(\mathbb{R})$ in the literature. Clearly, if $\Xi'\supseteq \Xi$, then
$\mathrm{AC}_\omega(\mathbb{R};\Xi')$ implies
$\mathrm{AC}_\omega(\mathbb{R};\Xi)$, so our inclusions of descriptive classes naturally give an
implication diagram for the axiom fragments as displayed in Figure \ref{fig:implications}. The diagram forms a
rectangular solid where the back side consists of the boldface principles, the front side of the lightface principles, the left side of the non-uniform principles, and the right side of the uniform principles. We shall state the main applications of our main theorem in terms of this diagram below.

Before doing so, we introduce the final axiom family: for any of our descriptive classes $\Gamma$, we write 
$\mathsf{DC}(\mathbb{R};\Gamma)$ for the \emph{Axiom of Dependent Choice for $\Gamma$ relations}, i.e.\ for any non-empty set $X$ of reals in $\Gamma$ and any total relation $R\subseteq X\times X$ (i.e.\ for all $x$ there is a $y$ such that $x\mathrel{R}y$) that is in $\Gamma$, there is a sequence
$\langle x_i : i \in  \omega \rangle \in X^\omega$ such that for all $i\in \omega$, we have $x_i\mathrel{R}x_{i+1}$.\footnote{Kanovei's definition of $\mathsf{DC}(\mathbb{R};\Gamma)$ is slightly different. However, he showed in \cite{kanovei1979descriptive} that Lemma \ref{lem:1} and Theorem \ref{thm:kanovei} are also true for our definition of $\mathsf{DC}(\mathbb{R};\Gamma)$.}

\begin{lemma}[Kanovei, \cite{kanovei1979descriptive}]\label{lem:1}
If $\Gamma$ is one of the classes $\boldsymbol{\Pi}^1_n$ or $\Pi^1_n$, then 
$\mathsf{DC}(\mathbb{R};\Gamma)$ implies $\mathsf{AC}_\omega(\mathbb{R};\mathrm{unif}\Gamma)$.
\end{lemma}

We can now formulate Kanovei's separation result from the introduction:

\begin{theorem}[Kanovei, \cite{kanovei1979descriptive}]\label{thm:kanovei}\label{theorem main theorem kanovei}
For each $n\geq 1$, there is a model of $\mathsf{ZF}+\mathsf{DC}(\mathbb{R};\boldsymbol{\Pi}^1_n)+
\neg\mathsf{AC}_\omega(\mathbb{R};\mathrm{unif}\Pi^1_{n+1})$.
\end{theorem}

Our main theorem is the following separation of the principles, which strengthens Kanovei's Theorem \ref{thm:kanovei}.

\begin{theorem}\label{theorem main theorem countable}\label{thm:main}
For every $n\geq 1$, there is a model of $\mathsf{ZF}+\mathsf{DC}(\mathbb{R};\boldsymbol{\Pi}^1_n)+
\neg\mathsf{AC}_\omega(\mathbb{R};\mathrm{unif}\Pi^1_{n+1})+\neg\mathsf{AC}_\omega(\mathbb{R};\mathbf{Ctbl})$.
\end{theorem}

We can represent this pictorially, with double arrows representing implications, and dashed lines representing the separation between layers due to Theorem \ref{thm:main}.

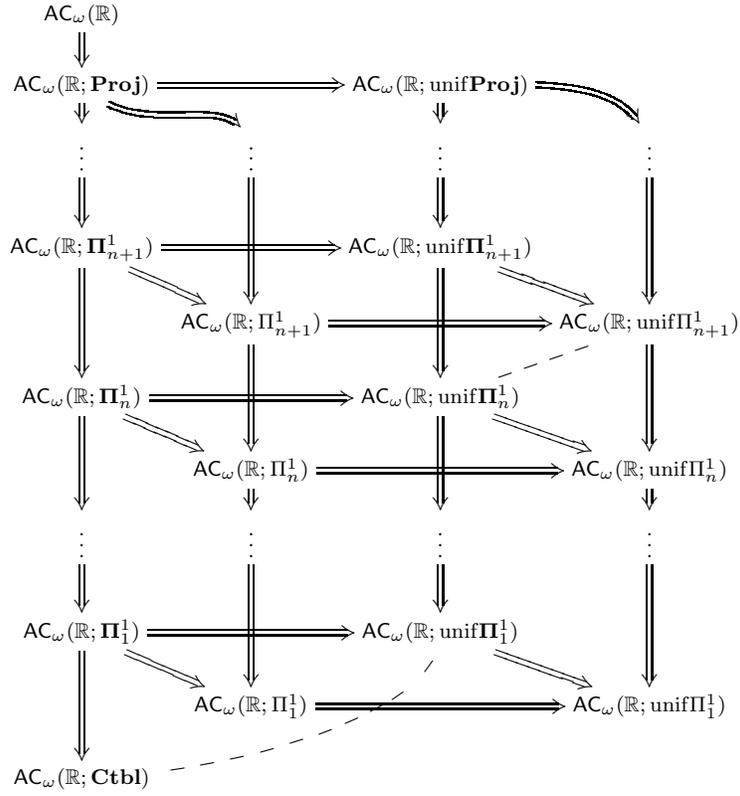
\begin{figure}{
{\scriptsize$$\xymatrix@C=0.5em@R=1.5em{
\mathsf{AC}_\omega(\mathbb{R}) \ar@{=>}[d]\\
\mathsf{AC}_\omega(\mathbb{R};\mathbf{Proj}) \ar@{=>}[rr]\ar@{=>}[d]\ar@{=>}@(dr,u)[rd] & &
\mathsf{AC}_\omega(\mathbb{R};\mathrm{unif}\mathbf{Proj}) \ar@{=>}[d]\ar@{=>}@(r,u)[dr] \\
\vdots\ar@{=>}[d] & \vdots\ar@{=>}[dd] & \vdots\ar@{=>}[d]& \vdots\ar@{=>}[dd]\\
\mathsf{AC}_\omega(\mathbb{R};\boldsymbol{\Pi}^1_{n+1}) \ar@{=>}[rr]\ar@{=>}[dd]\ar@{=>}[dr] & 
&
\mathsf{AC}_\omega(\mathbb{R};\mathrm{unif}\boldsymbol{\Pi}^1_{n+1}) \ar@{=>}[dr]\ar@{=>}[dd] &
\\
& 
\mathsf{AC}_\omega(\mathbb{R};\Pi^1_{n+1}) \ar@{=>}[rr]\ar@{=>}[dd] &
& 
\mathsf{AC}_\omega(\mathbb{R};\mathrm{unif}\Pi^1_{n+1}) \ar@{=>}[dd] \\
\mathsf{AC}_\omega(\mathbb{R};\boldsymbol{\Pi}^1_{n}) \ar@{=>}[rr]\ar@{=>}[dd]\ar@{=>}[dr] &
&
\mathsf{AC}_\omega(\mathbb{R};\mathrm{unif}\boldsymbol{\Pi}^1_{n}) \ar@{=>}[dr]\ar@{=>}[dd]\ar@{--}[ur]
 &
\\
&
\mathsf{AC}_\omega(\mathbb{R};\Pi^1_{n}) \ar@{=>}[rr]\ar@{=>}[d] &
&
\mathsf{AC}_\omega(\mathbb{R};\mathrm{unif}\Pi^1_{n}) \ar@{=>}[d] \\
\vdots\ar@{=>}[d] & \vdots\ar@{=>}[dd] & \vdots\ar@{=>}[d]& \vdots\ar@{=>}[dd]\\
\mathsf{AC}_\omega(\mathbb{R};\boldsymbol{\Pi}^1_{1}) \ar@{=>}[rr]\ar@{=>}[dd]\ar@{=>}[dr] &
&
\mathsf{AC}_\omega(\mathbb{R};\mathrm{unif}\boldsymbol{\Pi}^1_{1}) \ar@{=>}[dr] &
\\
&
\mathsf{AC}_\omega(\mathbb{R};\Pi^1_{1}) \ar@{=>}[rr] &
&
\mathsf{AC}_\omega(\mathbb{R};\mathrm{unif}\Pi^1_{1}) \\
\mathsf{AC}_\omega(\mathbb{R};\mathbf{Ctbl})\ar@{--}@(r,d)[rruu]&  &\\
}$$}}
\caption{Implication diagram of the fragments of $\mathsf{AC}_\omega(\mathbb{R})$ \label{fig:implications}}
\end{figure}

\subsection{Consequences of our main theorem}

We can now separate some of the sides of the solid quadrilateral in our Figure \ref{fig:implications}
by applying our Theorem \ref{thm:main}.

\begin{corollary}\label{cor:1}
There is a model of $\mathsf{ZF}+\mathsf{DC}(\mathbb{R};\mathbf{Proj})+
\neg\mathsf{AC}_\omega(\mathbb{R};\mathbf{Ctbl})$.
\end{corollary}

\begin{proof}
We argue by compactness. Let $T_\infty$ be the theory $\mathsf{ZF}$ along with the sentences $\mathsf{DC}(\mathbb{R};\boldsymbol{\Pi}^1_n)$ for each $n\in\omega$ and the sentence $\neg\mathsf{AC}_\omega(\mathbb{R};\mathbf{Ctbl})$. Let $F \subsetneq T_\infty$ be a finite subtheory of $T_\infty$. Then $F$ only contains finitely many sentences of the form $\mathsf{DC}(\mathbb{R};\boldsymbol{\Pi}^1_n)$. Let $k$ be the maximum $n$ such that $\mathsf{DC}(\mathbb{R};\boldsymbol{\Pi}^1_n) \in F$. By Theorem \ref{thm:main}, there is a model $N_k$ which satisfies $\mathsf{ZF}$ and $\mathsf{DC}(\mathbb{R};\boldsymbol{\Pi}^1_k)$, and so 
$N_k \models F$. Hence $T_\infty$ is finitely satisfiable. By compactness, there is a model of $\mathsf{ZF}+\mathsf{DC}(\mathbb{R};\mathbf{Proj})+\neg\mathsf{AC}_\omega(\mathbb{R};\mathbf{Ctbl})$.
\end{proof}

This, and Lemma \ref{lem:1}, yields two corollaries:
\begin{corollary}
For every $n\geq 1$, there is a model of $\mathsf{ZF}+\mathsf{AC}_\omega(\mathbb{R};\mathrm{unif}\boldsymbol{\Pi}^1_n)+
\neg 
\mathsf{AC}_\omega(\mathbb{R};\mathrm{unif}\Pi^1_{n+1})$.
\label{cor:2}
\end{corollary}

\begin{corollary}\label{cor:3}
There is a model of $\mathsf{ZF}+\mathsf{AC}_\omega(\mathbb{R};\mathrm{unif}\mathbf{Proj})+
\neg\mathsf{AC}_\omega(\mathbb{R};\mathbf{Ctbl})$.
\end{corollary}

Corollary \ref{cor:2} separates the horizontal slices of the diagram on the uniform side 
from each other. In the diagram, this is indicated by the dashed line between 
$\mathsf{AC}_\omega(\mathbb{R};\mathrm{unif}\boldsymbol{\Pi}^1_n)$ and
$\mathsf{AC}_\omega(\mathbb{R};\mathrm{unif}\Pi^1_{n+1})$, indicating that there is no implication from the lower horizontal slice to the higher slice.

Corollary \ref{cor:3} separates the boldface non-uniform part of the diagram from the uniform part by showing that the strongest uniform principle does not imply the weakest non-uniform boldface principle. This is indicated by the dashed line between 
$\mathsf{AC}_\omega(\mathbb{R};\mathbf{Ctbl})$ and 
$\mathsf{AC}_\omega(\mathbb{R};\mathrm{unif}\boldsymbol{\Pi}^1_1)$ in the diagram. So, in particular $\AC_\omega(\R;\mathbf{\Sigma}^1_n)$ implies but, is not implied by,  $\AC_\omega(\R;\mathsf{unif}\mathbf{\Sigma}^1_n)$.

Furthermore, by the remark that $\mathbf{Ctbl}\subseteq\mathbf{F}_\sigma\subseteq
\boldsymbol{\Sigma}^1_1$, our results separate all non-trivial non-uniform Borel
descriptive choice principles from the uniform descriptive choice principles. 

Any weaker Borel descriptive choice principles are provable in $\mathsf{ZF}$:

\begin{proposition}[Folklore]
The descriptive choice principle 
$\mathsf{AC}_\omega(\mathbb{R};\mathbf{F}_\sigma\cap\mathbf{G}_\delta)$ is provable in $\mathsf{ZF}$.
\end{proposition}

\begin{proof}
We can canonically pick an element from any $\mathbf{F}_\sigma\cap\mathbf{G}_\delta$ set like so: 
by Hausdorff's Difference Lemma, every $\mathbf{F}_\sigma \cap \mathbf{G}_\delta$ set $A$ of reals can be written as an $\alpha$-difference of closed sets for some $\alpha < \omega_1$, i.e.\ there is a sequence $\langle C_\beta : \beta < \alpha \rangle$ of closed sets of reals such that
\begin{align*}
    A= \mathrm{Diff}_{\beta<\alpha}C_\beta:=\{x \in \bigcup_{\beta < \alpha} C_\beta : \text{ the least } \beta < \alpha \text{ such that } x \notin C_\beta \text{ is odd}\}.
\end{align*}
With extra care, we can uniquely specify such a sequence (cf.\ \cite[3.E.1]{andretta2001notes}). Let $A\not=\emptyset$ be a $\mathbf{F}_\sigma\cap\mathbf{G}_\delta$  set of reals, and let $\langle C_\beta : \beta < \alpha \rangle$ be its sequence. So, it suffices to canonically define a real in $A$ from $\langle C_\beta : \beta < \alpha \rangle$.  There is a minimal even $\beta < \alpha$ such that  $C_\beta \supsetneq C_{\beta+1}$. Let $\{O_n : n \in \omega\}$ enumerate the basic open sets and let $n \in \omega$ be minimal such that $O_n \cap C_{\beta+1}= \emptyset$ and $O_n \cap C_\beta \neq \emptyset$. Since $O_n \cap C_\beta$ is closed, there is a unique pruned tree $T\subseteq \omega^{<\omega}$ such that $[T]= O_n \cap C_\beta$. The left-most branch of $T$ yields a real in $A$.
\end{proof}

\section{Model construction for the main theorem}

\subsection{Slicing forcing notions}\label{section slicing forcing notions}

Here, we describe $n$-slicing forcing notions. These are key to our separation results: in Section \ref{section proof main theorem}, we use $n$-slicing forcing notions to prove Theorem \ref{thm:main}. Before we can introduce $n$-slicing forcing notions, we need two further definitions, \emph{$n$-slice products}, and \emph{$n$-absoluteness for slices}. The former is a special variant of an $\omega$-product of forcing notions. The latter specifies a form of complexity-limited absoluteness for $n$-slice products.

Let $\xi \leq \omega_1$ be an ordinal and let $\langle \mathbb{P}_\nu : \nu < \xi \rangle$ be a sequence of forcing notions. We define the \emph{$\omega$-slice product of $\langle \mathbb{P}_\nu : \nu < \xi \rangle$ with finite support} as the set $\pz$ of all partial functions $p$ such that $\dom{p} \subseteq \xi \times \omega$ is finite and for every $(\nu,k) \in \xi \times \omega$, $p(\nu,k) \in \pz_\nu$ ordered by
\begin{align*}
    p \leq q :\iff \forall (\nu,k) \in \dom{q} (p(\nu,k) \leq q(\nu,k)).
\end{align*}
We say $\mathbb{P}$ has \emph{length} $\xi$. If, for every  $\nu,\nu' <\omega_1$, $\pz_\nu = \pz_{\nu'}$, then we say $\pz$ is the \emph{$\omega$-slice product of $\pz_\nu$} with finite support of length $\xi$. Let $I \subseteq \xi \times \omega$. For $A \subseteq \pz$, we define $A {\upharpoonright} I := \{p {\upharpoonright} I : p \in A\}$. Then for every $\pz$-generic filter $G$, $G {\upharpoonright} I$ is a $\pz {\upharpoonright} I$-generic filter.

Let $\xi = \omega_1$. We define $\mathrm{Aut}(\omega_1\times \omega, \omega)$ as the group of all bijections $\pi$ of $\omega_1 \times \omega$ such that for every $\nu \in \omega_1$, $\pi[\{\nu\}\times \omega]=\{\nu\}\times \omega$. Let $p \in \pz$ and let $\pi \in \mathrm{Aut}(\omega_1\times \omega, \omega)$. We define $\pi^\ast(p):=p'$ with $\dom{p'} = \pi[\dom{p}]$ and for $i\in \dom{p'}$, $p'(i) = p(\pi^{-1}(i))$. So, every automorphism $\pi \in \mathrm{Aut}(\omega_1\times \omega, \omega)$ induces an automorphism $\pi^\ast$ on $\pz$. Let $\mathcal{G} := \{\pi^\ast : \pi \in \mathrm{Aut}(\omega_1\times \omega, \omega)\}$ be the group of all such automorphisms on $\pz$.

We say $\mathcal{S} \subseteq \mathcal{P}( \omega_1 \times \omega)$ is a \emph{set of slices} if, for every $s \in \mathcal{S}$, $s$ is of the form $X \times \omega$, where $X \subseteq \omega_1$. Let $\mathcal{S} \subseteq \mathcal{P}( \omega_1 \times \omega)$ be a set of slices and for every $s \in \mathcal{S}$, let $H_s := \{\pi^\ast : \forall (\nu,k) \in s (\pi(\nu,k)=(\nu,k))\} \subseteq \mathcal{G}$ be the subgroup of all $\pi^\ast$ such that $\pi$ point-wise fixes $s$. We define $\mathcal{F}_\mathcal{S}$ as the filter on the subgroups of $\mathcal{G}$ generated by $\{H_s : s \in \mathcal{S}\}$. Then $\mathcal{F}_\mathcal{S}$ is normal. Let $G$ be a $\pz$-generic filter. We denote the symmetric extension we obtain from $G$ and $\mathcal{F}_\mathcal{S}$ by $V(G,\mathcal{S})$.

\begin{lemma}\label{lemma V[G|z] subset N}\label{lemma sets of ordinals in N}
Let $\langle \mathbb{P}_\nu : \nu < \omega_1 \rangle$ be a sequence of forcing notions, let $\pz$ be the $\omega$-slice product of $\langle \mathbb{P}_\nu : \nu < \omega_1 \rangle$ with finite support, let $\mathcal{S} \subseteq \mathcal{P}( \omega_1 \times \omega)$ be a set of slices, and let $G$ be a $\pz$-generic filter. Then 
\begin{enumerate}
    \item for every $s \in \mathcal{S}$, $V[G {\upharpoonright} s] \subseteq V(G,\mathcal{S})$ and
    \item for every set of ordinals $X$, $X \in V(G,\mathcal{S})$ iff there is some $s\in \mathcal{S}$ such that $X \in V[G{\upharpoonright}s]$.
\end{enumerate}
\end{lemma}

\begin{proof}
For the first part, let $X \in V[G{\upharpoonright}s]$. Then there is a $\pz{\upharpoonright}s$-name $\dot{X}$ such that $\dot{X}_{G{\upharpoonright}s}=X$. Since $\pz{\upharpoonright}s \subseteq \pz$, $\dot{X}$ is also a $\pz$-name and $\dot{X}_G =X$. Hence, it is enough to show that every $\pz {\upharpoonright} s$-name is hereditarily symmetric. Let $\sigma$ be a $\pz {\upharpoonright} s$-name, let $\pi^\ast \in H_s$, and let $(\tau,p)\in \sigma$. By assumption $p \in \pz {\upharpoonright}s$ and $\pi^\ast$ fixes $s$, so $\pi^\ast(p)=p$. Hence, $\pi^\ast(\sigma)=\{(\pi^\ast(\tau),\pi^\ast(p)) : (\tau,p)\in \sigma\}= \{(\pi^\ast(\tau),p) : (\tau,p)\in \sigma\}$. By induction, $\sigma$ is hereditarily symmetric.

We prove the second part. The backwards direction follows directly from the first part. For the forwards direction, let $\dot{X}$ be a symmetric $\pz$-name for $X$, let $s \in \mathcal{S}$ such that $H_s \subseteq \sym{\dot{X}}$, and let $p \in G$ such that $p \Vdash$ ``$\dot{X}$ is a set of ordinals''. We define a $\pz{\upharpoonright}s$-name $\sigma$ such that $\sigma_{G{\upharpoonright}s} = X$. Let $\sigma:=\{(\check{\xi}, q{\upharpoonright}s) : q \in \pz, $ and $q \leq p ,$ and $ q \Vdash \check{\xi} \in \dot{X}\}$ and let $X' := \sigma_{G{\upharpoonright}s}$.

First, we prove $X \subseteq X'$. Let $\xi \in X$. Then there is a $q \in G$ such that $(\check{\xi},q)\in \dot{X}$.  As $G$ is a filter, there is an $r \in G$ which witnesses that $p$ and $q$ are compatible. Hence, $r \leq p$ and $r \Vdash \check{\xi} \in \dot{X}$. Then $(\check{\xi}, r{\upharpoonright}s) \in \sigma$ and so $\xi \in X'$. 

Next, we prove $X' \subseteq X$. Let $\xi \in X'$. Then there is a $q \in G {\upharpoonright} s$ such that $(\check{\xi},q) \in \sigma$. Hence, there is an $r\leq p$ such that $q=r{\upharpoonright}s$ and $ r \Vdash \check{\xi} \in \dot{X}$. We suppose for a contradiction, that $q \not\Vdash \check{\xi}\in \dot{X}$. Then there is a $q' \leq q$ such that $q' \Vdash \check{\xi}\notin \dot{X}$. Hence, $q'$ and $r$ are incompatible, but agree on $s$ as functions. Let $\pi \in H_s$ such that $\dom{r} \cap \pi[\dom{q'}] \subseteq s$. Then $\pi^\ast(q')$ and $r$ are compatible, as they agree on their common domain. By construction $\pi^\ast \in H_s \subseteq \sym{\dot{X}}$, so $\pi^\ast(\dot{X})= \dot{X}$. Hence, $\pi^\ast(q') \Vdash \check{\xi}\notin \dot{X}$ which is a contradiction since $r$ and $\pi^\ast(q')$ are compatible. Therefore, $q \Vdash \check{\xi}\in \dot{X}$ and so $\xi \in X$.
\end{proof}

\begin{definition}
Let $\pz$ be an $\omega$-slice product with finite support of a sequence $\langle \mathbb{P}_\nu : \nu < \omega_1 \rangle$ of forcing notions, let $\mathcal{S} \subseteq \mathcal{P}(\omega_1 \times \omega)$ be a set of slices, and let $n\geq 1$. We say $\pz$ is \emph{$n$-absolute for $\mathcal{S}$-slices} if, for every $s \in \mathcal{S}$, every $\Sigma^1_n$ formula with real parameters $V[G {\upharpoonright} s]$ is absolute between $V[G {\upharpoonright} s]$ and $V[G]$. A set of slices $\mathcal{S} \subseteq \mathcal{P}( \omega_1 \times \omega)$ is \emph{unbounded} if, for every $s \in \mathcal{S}$, $\{\nu : \{\nu\}\times \omega \subseteq s \}$ is unbounded in $\omega_1$. We say $\pz$ is \emph{$n$-absolute for slices} if  for every unbounded set of slices $\mathcal{S}$, $\pz$ is $n$-absolute for $\mathcal{S}$-slices.
\end{definition}

Note that any $\omega$-slice product with finite support of length $\omega_1$ is 2-absolute for slices, by Shoenfield absoluteness \cite[Theorem 25.20]{jech2003set}.

\begin{definition}
A forcing notion $\pz$ is \emph{$n$-slicing} if there is a sequence of forcing notions $\langle \mathbb{P}_\nu : \nu < \omega_1 \rangle$ such that
\begin{enumerate}
    \item $\pz$ is the $\omega$-slice product of $\langle \mathbb{P}_\nu : \nu < \omega_1 \rangle$ with finite support,
    \item for every $\nu < \omega_1$, every $\pz_\nu$-generic filter $G$ over $V$ is uniquely determined by a real in $V[G] \setminus V$,
    \item $\pz$ is $n$-absolute for slices, and
    \item for every $\pz$-generic filter $G$, the set $\{(\ell,x_G^{(\ell,k)}) : (\ell,k)\in \omega^2\}$ is $\Pi^1_n$ in $V[G]$, where $x_G^{(\ell,k)}$ is the generic real defined by $G {\upharpoonright} \{(\ell,k)\}$.
\end{enumerate}
\end{definition}

\begin{lemma}[Sandwiching Lemma]\label{lemma n-absoluteness}
Let $n \in \omega$, let $\pz$ be $n$-slicing, let $G$ be a $\pz$-generic filter, and let $\mathcal{S} \subseteq \mathcal{P}( \omega_1 \times \omega)$ be an unbounded set of slices. For every $s \in \mathcal{S}$, every $\Sigma^1_n$-formula with real parameters in $V[G{\upharpoonright}s]$ is absolute between $V[G{\upharpoonright}s]$, $V[G]$, and $V(G,\mathcal{S})$. So every $\Sigma^1_n$-formula with real parameters in $V(G,\mathcal{S})$ is absolute between $V[G]$ and $V(G,\mathcal{S})$.
\end{lemma}

\begin{proof}
The second part follows from the first and Lemma \ref{lemma sets of ordinals in N}. The first part follows by induction, the base case is analytic absoluteness \cite[25.4]{jech2003set}.
\end{proof}

The following theorem will be proved in Section \ref{section n Jensen forcing}.

\begin{theorem}\label{theorem existence absolute for slices forcing notion}
Let $n \geq 2$. In $L$, there is an $n$-slicing forcing notion. 
\end{theorem}

\subsection{Theorem \ref{theorem existence absolute for slices forcing notion} implies the main theorem}\label{section proof main theorem}

In this section, we prove Theorem \ref{thm:main} under the assumption that Theorem \ref{theorem existence absolute for slices forcing notion} has been established. In particular, we construct our model from $L$ which then, by Theorem \ref{theorem existence absolute for slices forcing notion} has an $n$-slicing forcing notion for every $n \geq 2$.

We fix a natural number $n \geq 1$ and aim to construct a model of $\mathsf{ZF} + \mathsf{DC}(\mathbb{R};\boldsymbol{\Pi}^1_n) +
\neg \mathsf{AC}_\omega(\mathbb{R}; \mathrm{unif}\Pi^1_{n+1}) + \neg\mathsf{AC}_\omega(\mathbb{R};\mathbf{Ctbl})$. As $n+1 \geq 2$, let $\pz$ be an $(n+1)$-slicing forcing notion, let $G$ be a $\pz$-generic filter over $L$, let $Z:=\{(F \cup (\omega_1 {\setminus} \omega))\times \omega : F \subseteq \omega$ is finite$\}$, let $N:= L(G,Z)$, for every $\ell \in \omega$, let $A_\ell := \{x_G^{(\ell,k)} : k \in \omega\}$ and let $A:= \{A_\ell : \ell \in \omega\}$. Then, in $L[G]$, each $A_\ell$ is countable and $A$ is in $\mathrm{unif}\Pi^1_{n+1}$. We shall show that, in $N$, $A$ has no choice function, each $A_\ell$ is still countable, and $A$ is still in $\mathrm{unif}\Pi^1_{n+1}$. 

\begin{lemma}\label{lem: complexity A}
In $N$, $A$ is $\mathrm{unif}\Pi^1_{n+1}$ and for every $\ell \in \omega$, $A_\ell$ is countable.
\end{lemma}

\begin{proof}
For every $\ell \in \omega$, let $z:= (\{\ell\} \cup (\omega_1 \setminus \omega)) \times \omega$. Then $A_\ell$ is countable in $L[G{\upharpoonright} z]$ by the enumeration defined by $G {\upharpoonright} z$, hence $A_\ell$ is countable in $N$ by Lemma \ref{lemma V[G|z] subset N}. So we show the first part. Since $\pz$ is $(n+1)$-slicing, $\widehat{A}$ is $\Pi^1_{n+1}$ in $L[G]$. Let $\varphi$ be a $\Pi^1_{n+1}$ formula defining $\widehat{A}$. By Lemmas \ref{lemma V[G|z] subset N} and \ref{lemma n-absoluteness}, $N$ contains each $x_G^{(\ell,k)}$ and $\varphi$ is absolute between $L[G]$ and $N$. Therefore, $\varphi$ defines $\widehat{A}$ in $N$ and so $A$ is in $\mathrm{unif}\Pi^1_{n+1}$.
\end{proof}

\begin{proposition}\label{proposition AC fail in N}
In $N$, $\mathsf{AC}_\omega(\mathbb{R};\mathrm{unif}\Pi^1_{n+1})$ and $\mathsf{AC}_\omega(\mathbb{R};\mathbf{Ctbl})$ fail.
\end{proposition}

\begin{proof}
By Lemma \ref{lem: complexity A}, it suffices to show that $A$ has no choice function in $N$. Suppose there is a choice function $f:\omega \to \R$ for $A$ in $N$. Since $f$ is a countable sequence of reals, we can code $f$ as a real and so as a subset of $\omega$. By Lemma \ref{lemma sets of ordinals in N}, there is some $z \in Z$ such that $f\in L[G{\upharpoonright}z]$. Hence, there is a $\pz {\upharpoonright} z$-name $\dot{f}$ for $f$. As $z\in Z$,  there is an $\ell \in \omega$ such that $z$ does not meet any $\{\ell'\}\times \omega$ with $\ell' \geq \ell$. Let $\ell$ be minimal with that property and let $k \in \omega$ such that $f(\ell)=x^{(\ell,k)}_G$. Then there is a $p \in G$ such that $p \Vdash ``f$ is a choice function for $A" \land \dot{f}(\ell)=\dot{x}^{(\ell,k)}_G$, where $\dot{x}^{(\ell,k)}_G$ is the canonical $\pz$-name for $x^{(\ell,k)}_G$. Let $k' \neq k \in \omega$ such that $(\ell,k') \in (\{\ell\}\times \omega){\setminus}\dom{p}$ and let $\pi \in \mathrm{Aut}(\omega_1\times \omega, \omega)$ which only swaps $(\ell,k)$ and $(\ell,k')$. Then $\pi^\ast \in H_z \subseteq \sym{\dot{f}}$ and $p$ and $\pi^\ast(p)$ are compatible. But $\pi^\ast(p) \Vdash \dot{f}(\ell)=\pi^\ast(\dot{x}^{(\ell,k)}_G)=\dot{x}^{(\ell,k')}_G$, which is impossible as these are distinct.
\end{proof}

\begin{proposition}\label{proposition DC holds in N} 
In $N$, $\mathsf{DC}(\mathbb{R};\boldsymbol{\Pi}^1_n)$ holds. 
\end{proposition}

\begin{proof}
Let $X \subseteq \R$ be a $\mathbf{\Pi}^1_{n}$ set of reals in $N$ and let $R \subseteq (\R)^2$ be a total relation such that $R$ is $\mathbf{\Pi}^1_{n}$ in $N$. Then there are $\Pi^1_{n}$ formulas $\varphi$ and $\psi$ with real-parameter $a$ and $b$ in $N$ which define $X$ and $R$ in $N$, respectively. We can assume $a=b$. By Lemma \ref{lemma sets of ordinals in N}, there is a $z\in Z$ such that $a \in L[G{\upharpoonright}z]$. Let $X_\varphi$ and $R_\psi$ be the sets defined by $\varphi(a)$ and $\psi(a)$ in $L[G{\upharpoonright}z]$, and let $\chi(x,a)$ be the formulate $\varphi(x,a)\to\exists x'\psi((x,x'),a)$. Then $\chi(x,a)$ is $\Sigma^1_{n+1}$ and so absolute by Lemma \ref{lemma n-absoluteness}. By downwards-absoluteness, for every $x$,$ \chi(x,a)$ is true in $L[G{\upharpoonright}z]$.  As usual $L[G{\upharpoonright}z] \models \ZFC$, so there is a sequence $\langle x_i : i \in \omega \rangle$ in $L[G{\upharpoonright}z]$ such that $x_i \mathrel{R_\psi} x_{i+1}$ for every $i \in \omega$. By Lemma \ref{lemma V[G|z] subset N}, $\langle x_i : i \in \omega \rangle \in N$, and $\varphi(a)$ and $\psi(a)$ are absolute between $L[G {\upharpoonright}z]$ and $N$, so $x_i \mathrel{R} x_{i+1}$ in $N$ for every $n\in \omega$.  
\end{proof}

This finishes the proof of Theorem \ref{thm:main}. In the classical case, when aiming to construct a model of $\ZF + \neg\AC_\omega(\R;\mathbf{Ctbl})$, we construct a symmetric extension from an $\omega\times\omega$-product, where we fix finitely many rows from the $\omega \times \omega$-matrix of the product. Here, we modify this construction by instead using an $\omega_1\times\omega$-product. The extra $\omega_1$-many rows are used to prove Theorem \ref{theorem existence absolute for slices forcing notion} (via the Kanovei-Lyubetsky Lemma, Lemma \ref{theorem aabsoluteness}). We then `ignore' the later reals (i.e.\ those added by the $\alpha \times k$-component for $\alpha \in \omega_1\backslash\omega$). This motivates our choice of a filter, $\mathcal{F_S}$, which contains subgroups of permutations which can move any finite initial segment of $\omega_1$, but fix everything beyond $\omega$. Hence $A$ is then the countable $\omega\times\omega$-sized bottom corner of the $\omega_1 \times\omega$ grid, formed by `chopping off' the end $\omega_1\backslash \omega$ tail of each $\omega_1$-length sequence of reals. 

\section{Jensen forcing and $n$-Jensen forcing}\label{sec: Jensen}

\subsection{Jensen forcing}\label{section jensen forcing}

The aim of Section \ref{sec: Jensen} is to prove Theorem \ref{theorem existence absolute for slices forcing notion}, i.e.\ to prove that for every $n \geq 2$, there is an $n$-slicing forcing notion in $L$. In preparation for this, here we briefly introduce Jensen forcing, which we will then generalise in Section \ref{section n Jensen forcing} to construct $n$-slicing forcing notions. Our presentation and notation are based on that of \cite{friedman2018model}.

Jensen forcing was originally introduced by Jensen in \cite{jensen} (see also \cite[\S28]{jech2003set}) to construct a model with especially simple non-constructible reals. After Cohen's totemic proof of the consistency of $\ZFC$ with the existence of non-constructible reals, it was naturally asked how simple these non-constructible reals could be. Shoenfield's absoluteness theorem showed that the minimum possible complexity was $\mathbf{\Delta}^0_3$. Jensen used Jensen forcing to construct such a model, where there was a $\mathbf{\Delta}^0_3$ real, with the minimal non-trivial degree of constructibility (if $x \in L[a]$ and $a \not\in L[x]$, then $x \in L$). We are interested in Jensen forcing because its generic reals are unique (i.e.\ there is only one Jensen real in the extension) and the set of its generic reals over $L$ is $\Pi^1_2$. Therefore, products of Jensen forcing can be used to add a set of complexity $\Pi^1_2$ in a controlled way. To be able to add sets of higher complexity, we generalise Jensen's construction of Jensen forcing in Section \ref{section n Jensen forcing}.

Before we can begin with the construction of Jensen forcing, we need some additional notation for trees. Let $\sz$ be Sacks forcing, let $T \in \sz$, and let $t \in T$. We define $T_t := \{s \in T : s \subseteq t \lor t \subseteq s\}$. Then $T_t \in \sz$ and every node in $T_t$ is compatible with $t$.  Let $\cz:=\{(2^{<\omega})_t : t \in 2^{<\omega}\}$, ordered by inclusion, be the arboreal version of Cohen forcing. The \emph{meet} of two perfect trees, $S$ and $T$, denoted $S \land T$, is the largest perfect tree contained in $S \cap T$, if it exists (see \cite[pg. 5]{friedman2018model}). We say a subset $P$ of $\sz$ is a \emph{perfect poset} if, $\cz \subseteq P$,  $P$ is closed under finite unions, and for every $S,T \in P$, if $S$ and $T$ are compatible in $\sz$, then $S \land T \in P$.  If $P$ is a perfect poset, we order $P \times \omega$ by 
\begin{align*}
    (S,n)\leq (T,m) :\iff S \leq T \land m \leq n \land 2^m \cap S = 2^m \cap T.
\end{align*}
We denote the poset $(P \times \omega, \leq)$ by $\qz(P)$.

\begin{definition}[Jensen's Operation, $P^H$]\label{definition Jensen star operation}
Let $M$ be a countable transitive model (ctm) of $\ZFC^-+ `` \mathcal{P}(\omega)$ exists$"$, let $P \in M$ be a perfect poset, let $\qz(P)^{<\omega}$ be the $\omega$-product of $\qz(P)$ with finite support, let $H$ be a $\qz(P)^{<\omega}$-generic filter over $M$, and let $\langle \mathcal{T}_k : k \in \omega \rangle$ be the sequence of generic trees added by $H$. We define $P^H$ as the closure of 
\begin{align*}
    P\cup \{\mathcal{T}_k \land S : S \in P \text{, } k \in \omega \text{, and } \mathcal{T}_k \land S \neq \emptyset\}
\end{align*}
under finite unions, and order $P^H$ by inclusion.
\end{definition}

Jensen's operation expands $P$ to a new poset $P^H$, which has many nice properties. Crucially, predense subsets of $P$ remain predense in $P^H$:

\begin{lemma}\label{lemma facts Jensen forcing}
Let $M$, $P$, $H$, and $\langle \mathcal{T}_k : k \in \omega \rangle$ be as in Definition \ref{definition Jensen star operation}. Then:
\begin{enumerate}
    \item $P^H$ is a perfect poset.
    \item $\{\mathcal{T}_k : k \in \omega\}$ is a maximal antichain in $P^H$.
    \item Every predense set $D \subseteq P$ in $M$ remains predense in $P^H$.
\end{enumerate}
Let $P^{<\omega}$ and $(P^H)^{<\omega}$ be the finite support $\omega$-product of $P$ and $P^H$, respectively. Then:
\begin{enumerate}
    \item[4.] Every predense set $D \subseteq  P^{<\omega}$ in $M$ remains predense in $(P^H)^{<\omega}$.
\end{enumerate}
\end{lemma}

\begin{proof}
See \cite[Propositions 2.4 \& 2.5]{friedman2018model}.
\end{proof}

Nowadays, Jensen forcing is usually constructed using a $\diamondsuit$-sequence (cf.\ e.g.\ \cite[Theorem 28.1]{jech2003set}). However, when Jensen forcing was introduced, the $\diamondsuit$-principle had not been defined yet and so Jensen had to do the construction by hand. We follow Jensen's approach and do not use a $\diamondsuit$-sequence because then the construction is easier to generalise. Jensen forcing is defined as the union of a sequence $\langle P_\xi : \xi < \omega_1 \rangle$ of perfect posets. We define the sequence recursively in $L$: let $P_0$ be the closure of $\mathbb{C}$ under finite unions and let $\gamma_0$ be the least countable ordinal such that $L_{\gamma_0}$ is a model of $\mathsf{ZFC}^- + ``\mathcal{P}(\omega)$ exits'' containing $P_0$. We assume that for $\xi' < \xi$, we have already defined a pair $(L_{\gamma_{\xi'}} , P_{\xi'})$ such that $L_{\gamma_{\xi'}}$ is a model of $\mathsf{ZFC}^- + ``\mathcal{P}(\omega)$ exits'' and $P_{\xi'}$ is a perfect poset in $L_{\gamma_{\xi'}}$. If $\xi$ is a limit ordinal, then we set $P_\xi := \bigcup_{\xi'< \xi} P_{\xi'}$. If otherwise $\xi =\xi'+1$, we set $P_\xi := (P_{\xi'})^H$, where $H$ is the $<_L$-least $\mathbb{Q}(P_{\xi'})^{<\omega}$-generic filter over $L_{\gamma_{\xi'}}$. In both cases, let $\gamma_\xi$ be the least countable ordinal such that $P_{\xi} \in L_{\gamma_{\xi}}$ and $\mathbb{R} \cap L_{\gamma_{\xi}+1} \nsubseteq L_{\gamma_{\xi}}$. Finally, we define Jensen forcing as $\mathbb{J}:=\bigcup_{\xi < \omega_1} P_\xi$. Note that for every $\mathbb{J}$-generic filter $G$, $\bigcap \{[T] : T \in G\}$ is a real. We call such reals \emph{Jensen reals}.

\begin{theorem}[Jensen, \cite{jensen}]\label{theorem jensen original}
Jensen forcing is ccc, the set of Jensen reals over $L$ is $\Pi^1_2$ in every inner model of $\mathsf{ZF}$ containing $L$, and for every $\mathbb{J}$-generic filter $G$ over $L$, the set of Jensen reals in $L[G]$ is a singleton.
\end{theorem} 

\subsection{$n$-Jensen forcing notions}\label{section n Jensen forcing}

In this section, we define a certain iteration of Jensen's operation, a \textit{Jensen-sequence}, and then prove that each forcing notion associated to a Jensen-sequences, which we call a \emph{Jensen-like forcing notion}, exhibits generalisations of the properties Jensen forcing.  Then we construct a forcing notion to prove Theorem \ref{theorem existence absolute for slices forcing notion}, using  an analogue of \cite[Theorem 13]{KanoveiLyubetsky}, which we call the \emph{Kanovei-Lyubetsky Lemma} (Lemma \ref{theorem aabsoluteness}) by using a witnessing property for the underlying Jensen-sequence, \emph{$n$-completeness}, to ensure that the products of resulting Jensen-like forcing notions are $n$-absolute for slices, and hence can be used to separate the required fragments of choice.

First, we generalise Jensen's construction from Section \ref{section jensen forcing}. 

\begin{definition}
Let $\zeta \leq \omega_1$. We say $\langle(L_{\gamma_\xi},P_\xi) : \xi < \zeta \rangle$ is a \emph{Jensen-sequence} if, for every $\xi < \zeta$,
\begin{enumerate}
    \item $P_0$ is the closure of $\cz$ under finite unions,
    \item $\gamma_\xi$ is a countable ordinal such that $L_{\gamma_\xi} \models \ZFC^{-} +``\mathcal{P}(\omega)$ exists'' and $\mathbb{R} \cap \mathrm{L}_{\gamma_\xi+1} \nsubseteq \mathrm{L}_{\gamma_\xi}$,
    \item $P_\xi$ is a perfect poset in $L_{\gamma_\xi}$,
    \item if $\xi = \zeta +1$, then there is a $\qz(P_\zeta)^{<\omega}$-generic filter $H_\zeta \in L_{\gamma_{\xi}}$ over $L_{\gamma_\zeta}$ such that $P_\xi = (P_\zeta)^{H_\zeta}$, and
    \item $\langle P_\xi : \xi \in\zeta \rangle$ is continuous at limits, i.e.\ if $\xi$ is a limit, then $P_\xi = \bigcup_{\xi' < \xi}P_{\xi'}$.
\end{enumerate}
\end{definition}

As in the construction of Jensen forcing in Section \ref{section jensen forcing}, we can construct Jensen-sequences of length $\omega_1$ by recursion. Whilst Jensen uses the $<_L$-least generic filter, we consider \textit{all} possible forcing notions which are defined by Jensen-sequences, using any generic filter. We call a forcing notion $\pz$ \emph{Jensen-like} if there is a Jensen-sequence $\langle(L_{\gamma_\xi},P_\xi) : \xi < \omega_1 \rangle$ such that $\pz:= \bigcup_{\xi < \omega_1} P_\xi$. Note that, as with Jensen forcing, if $G$ is a generic filter for $\pz$, then $\bigcap \{[T] : T \in G\}$ is a real. We call such reals \emph{$\pz$-generic reals}.

We show that the properties of Jensen-like forcing notions generalise those of Jensen forcing in Theorem \ref{theorem jensen original}. Obviously, predense sets remain predense at successor steps along a Jensen-sequence as those steps are instances of Jensen's operation; moreover predense sets remain predense at limit steps and for the union of the sequence, and suitable products along a Jensen-sequence are ccc:

\begin{proposition}\label{proposition Jensen predense sets remain predense}
If $\pz$ is a Jensen-like forcing notion, then for every $\xi < \omega$, every predense set $D \subseteq P_\xi$ (or $P_\xi^{<\omega}$) in $L_{\gamma_\xi}$ is predense in $\pz$ (or $\pz^{<\omega}$).
\end{proposition}

\begin{proof}
Let $\langle(L_{\gamma_\xi},P_\xi) : \xi < \omega_1 \rangle$ be a Jensen-sequence such that $\pz:= \bigcup_{\xi < \omega_1} P_\xi$. Suppose that there is some $p \in \pz$ where for all $d \in D$, $d \bot p$. By definition, there is some $\xi < \omega_1$ such that $p \in P_\xi$. Then $D$ is not predense in $P_\xi$. Let $\zeta$ be minimal such that $\xi < \zeta$ and $D$ is not predense in $P_{\zeta}$. Then $\zeta$ is a successor. Since $\xi < \zeta$, $D \in L_{\gamma_\xi} \subseteq L_{\gamma_{\zeta-1}}$. But this contradicts Lemma \ref{lemma facts Jensen forcing}. The case $D \subseteq P_\xi^{<\omega}$ is similar.
\end{proof}

\begin{proposition}\label{proposition omega-product jensen ccc}
If $\pz$ is a Jensen-like forcing notion, then the $\omega$-product $\pz^{<\omega}$ with finite support is ccc in $L$.
\end{proposition}

\begin{proof}
By a condensation argument we can reduce this to Proposition \ref{proposition Jensen predense sets remain predense} (cf.\ \cite[Lemma 6]{jensen}).
\end{proof}

\begin{corollary}\label{corollary P has the c.c.c}
Let $\pz$ be a Jensen-like forcing notion, and $\mathbb{Q}$ be the $\omega$-slice product of $\pz$ with finite support of length $\omega_1$. Then for every set $I \subseteq \omega_1 \times \omega$, $\mathbb{Q}{\upharpoonright}I$ is ccc in $L$.
\end{corollary}

\begin{proof}
This follows from Proposition \ref{proposition omega-product jensen ccc} and a $\Delta$-system argument.
\end{proof}

Next, we consider the uniqueness of generic reals for Jensen-like forcing notions. With an argument similar to that of Jensen forcing, one can show that for every generic filter $G$ for a Jensen-like forcing notion $\pz$, the set of $\pz$-generic reals in $L[G]$ is a singleton. However, we are more interested in what happens after forcing with a product of $\pz$. Kanovei and Lyubetsky \cite{MR3691702} extended the uniqueness of Jensen reals to countable products of Jensen forcing. We generalise their result to Jensen-like forcing notions.

\begin{lemma}[Kanovei-Lyubetsky]\label{lemma uniqueness jensen generics}
Let $M$ be a ctm of $\ZFC^-+ `` \mathcal{P}(\omega)$ exists$"$, $P \in M$ be a perfect poset, and let $\dot{y} \in M$ be a $P^{<\omega}$-name for a real such that for every $k \in \omega$, $1_{P^{<\omega}} \Vdash \dot{y} \neq \dot{x}^k_G$. Then for every $\qz(P)^{<\omega}$-filter $H$ over $M$ and every generic tree $\mathcal{T}_k$ added by $H$, in $M[H]$ the set of conditions forcing that $\dot{y} \notin [\mathcal{T}_k]$ is dense in $(P^H)^{<\omega}$.
\end{lemma} 

\begin{proof}
See \cite[Theorem 3.1]{friedman2018model}.
\end{proof}

\begin{proposition}\label{proposition uniqueness jensen generics}
Let $\pz$ be a Jensen-like forcing notion and let $G$ be a $\pz^{<\omega}$-generic filter over $L$. Then for every $\pz$-generic real $y \in L[G]$ over $L$, there is some $k \in \omega$ such that $y = x_G^k$.
\end{proposition}

\begin{proof}
By a condensation argument we can reduce this to Lemma \ref{lemma uniqueness jensen generics}.
\end{proof}

\begin{corollary}\label{corollary Jensen uniqueness}
Let $\pz$ be a Jensen-like forcing notion and let $\mathbb{Q}$ be the $\omega$-slice product of $\pz$ with finite support of length $\omega_1$, let $I \subseteq \omega_1 \times \omega$, and let $G$ be a $\mathbb{Q} {\upharpoonright} I$-generic filter over $L$. Then for every $\pz$-generic real $y \in L[G]$ over $L$, there is some $i \in I$ such that $y = x^i_G$.
\end{corollary} 

\begin{proof}
Let $y \in L[G]$ be $\pz$-generic over $L$ and let $\dot{y}$ be a nice $\mathbb{Q} {\upharpoonright} I$-name for $y$. By Corollary \ref{corollary P has the c.c.c}, $\mathbb{Q} {\upharpoonright} I$ is ccc is and so $\dot{y}$ is countable. Hence, there is a countable infinite $I' \subseteq I$ such that $\dot{y}$ is a $\mathbb{Q} {\upharpoonright} I'$-name. Since $I'$ is countable, $\mathbb{Q} {\upharpoonright} I'$ is order isomorphic to $\pz^{<\omega}$. Hence, we can apply Proposition \ref{proposition uniqueness jensen generics} and so there is an $i \in I'$ such that $y = x^i_G$.
\end{proof}

Hence Jensen-like forcing notions satisfy generalisations of two of the three properties in Theorem \ref{theorem jensen original}.  The following proposition helps us with the remaining one, to determine the complexity of the set of generic reals.

\begin{proposition}\label{proposition characterization Jensen generics}
Let $\pz$ be a Jensen-like forcing notion, let $\langle(L_{\gamma_\xi},P_\xi) : \xi < \omega_1 \rangle$ be a Jensen-sequence with $\pz:= \bigcup_{\xi < \omega_1} P_\xi$, and let $\langle \mathcal{T}_k^\xi : k \in \omega \rangle$ be the $\qz(P_\xi)^{<\omega}$-generic sequence used to construct $P_{\xi+1}$. The following are equivalent:
\begin{enumerate}
    \item A real $x$ is $\pz$-generic over $L$,
    \item for every $\xi < \omega_1$, $x$ is $P_\xi$-generic over $L_{\gamma_\xi}$, and
    \item for every $\xi < \omega_1$, there is some $k \in \omega$ such that $x \in [\mathcal{T}^\xi_k]$.
\end{enumerate}
\end{proposition}

\begin{proof}
We start with $(1. \Rightarrow 3.)$. Let $\xi <\omega_1$. By Lemma \ref{lemma facts Jensen forcing}, $\{\mathcal{T}_k^\xi : k \in \omega \}$ is a maximal anitchain in $P_{\xi+1}$ and by Proposition \ref{proposition Jensen predense sets remain predense}, it remains predense in $\pz$. Hence, there is some $k \in \omega$ such that $x \in [\mathcal{T}^\xi_k]$.

For $(3. \Rightarrow 2.)$, let $\xi < \omega_1$ and let $D \subseteq P_\xi$ be dense in $L_{\gamma_\xi}$. By assumption, there is some $k \in \omega$ such that $x \in [\mathcal{T}^\xi_k]$. We define 
\begin{align*}
    E:= \{q \in \qz(P_\xi)^{<\omega} : q(k)=(T,n) \land \forall s \in 2^n \cap T (T_s \in D)\}.
\end{align*}
First, we prove that $E$ is dense in $\qz(P_\xi)^{<\omega}$: let $q \in \qz(P_\xi)^{<\omega}$ and let $q(k)=(T,n)$. For every $s \in 2^n \cap T$, $T_s \in P_\xi$, so there is an $S_s \in D$ with $S_s \leq T_s$. Let $S:= \bigcup_{s \in 2^n \cap T} S_s$ and $q'\in \qz(P_\xi)^{<\omega} $ by $q'(k) := (S,n)$ and for $l \neq k$, $q'(l) := q(l)$. Then $(S,n) \leq (T,n)$ and so $q' \leq q$ and $q' \in E$. Hence, $E$ is dense in $\qz(P_\xi)^{<\omega}$. Let $H_\xi$ be the $\qz(P_\xi)^{<\omega}$-generic filter corresponding to $\langle \mathcal{T}_k^\xi : k \in \omega \rangle$. Then $E$ meets $H_\xi$, so there is a pair $(T,n) \in \qz(P_\xi)$ such that for every $s \in 2^n \cap T$, $T_s \in D$ and $\mathcal{T}_k^\xi \subseteq T$. Let $s \in 2^n \cap T$ such that $s \subseteq x$. Then $x \in [T_s]$ and $T_s \in D$. Thus, $x$ is $P_\xi$-generic over $L_{\gamma_\xi}$.

Finally, for $(2. \Rightarrow 1.)$, let $\mathcal{A} \subseteq \pz$ be a maximal antichain. Since $\pz$ is ccc, $\mathcal{A}$ is countable, so let $\xi <\omega_1$ be such that $\mathcal{A} \subseteq P_\xi$ and $\mathcal{A} \in L_{\gamma_\xi}$. By assumption $x$ is $P_\xi$-generic over $L_{\gamma_\xi}$, so there is a $T \in P_\xi \subseteq P$ with $x \in [T]$.
\end{proof}

\begin{corollary}\label{cor Jensen complexity set of generics}
    Let $\pz$ be a Jensen-like forcing notion, let $J := \langle(L_{\gamma_\xi},P_\xi) : \xi < \omega_1 \rangle$ be a Jensen-sequence with $\pz:= \bigcup_{\xi < \omega_1} P_\xi$, and let $n > 2$. If there is an unbounded subsequence $J':=\langle (M_\xi,Q_\xi): \xi < \omega_1 \rangle$ of $J$ which is $\Delta^{\HC}_{n-1}$, then the set of all $\pz$-generics over $L$ is $\Pi^{1}_n$ in every inner model of $\mathsf{ZFC}$ containing $L$.
\end{corollary}
 
\begin{proof}
By Proposition \ref{proposition characterization Jensen generics}, a real $x$ is $\pz$-generic over $L$ iff for every $\xi < \omega_1$, $x$ is $P_\xi$-generic over $L_{\gamma_\xi}$. If $x$ is $P_\xi$-generic over $L_{\gamma_\xi}$, then $x$ is also $P_{\xi'}$-generic over $L_{\gamma_{\xi'}}$ for every $\xi' < \xi$. Hence, $x$ is $\pz$-generic over $L$ iff for every $\xi < \omega_1$, $x$ is $Q_\xi$-generic over $M_\xi$. Let $M$ be an inner model of $\mathsf{ZFC}$ containing $L$.  Note that $\omega_1^L$ and $L_{\omega_1^L}$ are $\Sigma^{\HC}_1$ in $M$. Hence, $J'$ is $\Delta^{\HC}_{n-1}$ in $M$ and so the set of $\pz$-generic reals over $L$ is $\Pi^{\HC}_{n-1}$ in $M$. By \cite[Lemma 25.25]{jech2003set}, the set of $\pz$-generic reals over $L$ is $\Pi^1_n$.
\end{proof}

Next, we aim to show that $\omega$-slice products of Jensen-like forcing notions are $n$-absolute for slices. Kanovei and Lyubetsky proved  an analogous result for almost disjoint forcing (\cite[$\S4$]{KanoveiLyubetsky}), building their product from an increasing sequence of forcing notions. Our method is somewhat analogous; instead of almost disjoint forcing, we use Jensen-like forcing notions. We first store together all of the Jensen-sequences in a single set.

\begin{definition}[$\mathcal{M}_{\mathcal{J}}$ and $\preccurlyeq$]\label{definition m-j-ast}
Let $\mathcal{M}_{\mathcal{J}}$ be the set of all $(M, P)$ such that there is a Jensen-sequence $J \in M$, such that $J^\frown (M,P)$ is a Jensen-sequence. We define an ordering $\preccurlyeq\,\, \subseteq \mathcal{M}_{\mathcal{J}} \times \mathcal{M}_{\mathcal{J}}$ by $(N,Q) \preccurlyeq (M,P)$ iff $N \subseteq M$, and either $Q = P$ or else there is a Jensen-sequence $J \in M$ such that
\begin{enumerate}
    \item there is some $\xi \in \dom{J}$ such that $J(\xi) = (N,Q)$ and
    \item $J^\frown (M,P)$ is a Jensen-sequence.
\end{enumerate}

Let $\vartheta \leq \omega_1$. A $\preccurlyeq$-increasing sequence $\langle (M_\xi, P_\xi) : \xi < \vartheta \rangle$ in $\mathcal{M}_{\mathcal{J}}$ is \emph{strict} if $M_{\xi} \subsetneq M_{\xi+1}$ and $P_\xi \subsetneq P_{\xi+1}$ for every $\xi < \vartheta$. 
\end{definition}

\begin{lemma}\label{lemma blow up to Jensen-sequence}
Let $\langle (M_\xi, P_\xi) : \xi < \omega_1 \rangle$ be a strictly $\preccurlyeq$-increasing sequence where $\langle P_\xi : \xi < \omega_1 \rangle$ is continuous at limits. Then $\bigcup_{\xi < \omega_1} P_\xi$ is a Jensen-like forcing notion.
\end{lemma}

\begin{proof}
It is enough to show that there is a Jensen-sequence which contains $\langle (M_\xi, P_\xi) : \xi < \omega_1 \rangle$ as an unbounded subsequence. We build a sequence $\langle J_\zeta : \zeta \leq \omega_1 \rangle$ of Jensen-sequences such that for every $0 < \zeta \leq \omega_1$,
\begin{enumerate}
    \item $\langle (M_{\xi}, P_{\xi}) : \xi < \zeta \rangle$ is an unbounded subsequence of $J_\zeta$,
    \item if $\zeta' < \zeta < \omega_1$, then $J_\zeta$ extends $J_{\zeta'}$, and
    \item if $\zeta \leq \theta$ is a limit, then $J_\zeta = \bigcup_{\zeta' < \zeta} J_{\zeta'}$.
\end{enumerate}
Then $J_{\omega_1}$ is the desired Jensen-sequence. We define the $J_\zeta$ recursively. By definition, there is a Jensen-sequence $J_0 \in M_0$, such that ${J_0}^\frown (M_0,P_0)$ is a Jensen-sequence. We set $J_1 := {J_0}^\frown (M_0,P_0)$. Suppose $J_{\zeta'}$ is already defined for $\zeta' < \zeta$. We distinguish cases:

Case 1: $\zeta = \zeta' +1$ and $\zeta' = \zeta'' + 1$. Then $\dom{J_{\zeta'}}$ is a successor ordinal $\eta = \eta'+1$ and $J_{\zeta'}(\eta') = (M_{\zeta''}, P_{\zeta''})$. By construction $(M_{\zeta''}, P_{\zeta''}) \preccurlyeq (M_{\zeta'}, P_{\zeta'})$ and the sequence is strict, so there is a Jensen-sequence $J'_{\zeta} \in M_{\zeta'}$ such that there is some $\xi \in \dom{J'_{\zeta}}$ with $J'_{\zeta}(\xi)= (M_{\zeta''}, P_{\zeta''})$ where ${J'_{\zeta}}^\frown (M_{\zeta'}, P_{\zeta'})$ is a Jensen-sequence. Let $J''_{\zeta}$ be the restriction of $J'_{\zeta}$ to $\dom{J'_{\zeta}}{\setminus} (\xi+1)$. We set $J_\zeta := {J_{\zeta'}}^\frown {J''_{\zeta}}^\frown (M_{\zeta'}, P_{\zeta'})$.

Case 2: $\zeta = \zeta' +1$ and $\zeta'$ is a limit. We set $J_\zeta := {J_{\zeta'}}^\frown (M_{\zeta'}, P_{\zeta'})$. Since $(M_{\xi}, P_{\xi}) : \xi < \zeta' \rangle$ is an unbounded subsequence of $J_{\zeta'}$ and $P_{\zeta'} = \bigcup_{\xi < \zeta'}P_{\xi}$, $J_\zeta$ is a Jensen-sequence.

Case 3: $\zeta$ is a limit. We set $J_\zeta := \bigcup_{\zeta' < \zeta} J_{\zeta'}$. 
\end{proof}

Only certain Jensen-sequences suffice for generating Jensen-like forcing notions whose $\omega$-slice products are $n$-absolute for slices. To this end, we define a kind of Jensen-sequence which will suffice, in analogy to \cite[Definition 15]{KanoveiLyubetsky}.

\begin{definition}
Let $n>2$. A strictly $\preccurlyeq$-increasing sequence $\langle (M_\xi, P_\xi) : \xi < \omega_1 \rangle$ is called \emph{$n$-complete} if, for every $\Sigma^{\HC}_{n-2}$ set $D \subseteq \mathcal{M}_{\mathcal{J}}$, there is a $\xi < \omega_1$ such that either $(M_\xi,P_\xi)\in D$ or there is no $(N,Q) \in D$ extending $(M_\xi,P_\xi)$.
 \end{definition}
 
 \begin{lemma}\label{lemma existence n complete sequence}
For every $n > 2$, there is a $\Delta^{\HC}_{n-1}$, $n$-complete strictly $\preccurlyeq$-increasing sequence $\langle (M_\xi, P_\xi) : \xi < \omega_1 \rangle$ in $L$ .
 \end{lemma}
 
\begin{proof}
Let $n>2$, and let $\Gamma \subset \omega_1 \times \HC$ be a universal $\Sigma^{\HC}_{n-2}$  set.  We define the required sequence recursively. Let $(M_0,P_0)$ be the $<_L$-least pair such that $(M_0, P_0) \in \mathcal{M}_{\mathcal{J}}$. Suppose that $\langle (M_{\xi'}, P_{\xi'}) : \xi' < \xi \rangle$ is already defined. If $\xi$ is a limit, then we set $P_\xi := \bigcup_{\xi' < \xi} P_{\xi'}$ and let $M_\xi$ be the $<_L$-least ctm of $\ZFCM$ such that $(M_\xi, P_\xi) \in \mathcal{M}_{\mathcal{J}}$ and $M_\xi$ contains $\langle (M_{\xi'}, P_{\xi'}) : \xi' < \xi \rangle$. If $\xi=\xi'+1$ is a successor, let $(M_\xi,P_\xi)$ be the $<_L$-least pair such that
\begin{enumerate}
    \item $(M_{\xi'},P_{\xi'})$ is strictly $\preccurlyeq (M_\xi,P_\xi)$,

    \item either $(M_\xi,P_\xi)\in D_{\xi'}:= \{m \in \mathcal{M} : (\xi', m) \in \Gamma\}$ or there is no $(N,Q) \in D_{\xi'}$ extending $(M_\xi,P_\xi)$.
\end{enumerate}
Note that ${<_L}\cap \HC^2$ is $\Delta^{\HC}_1$. By definition $\mathcal{M}_{\mathcal{J}}$ and $\preccurlyeq$ are $\Delta^{\HC}_1$, and $\Gamma$ is $\Sigma^{\HC}_{n-2}$, so $\langle (M_\xi, P_\xi) : \xi < \omega_1 \rangle$ is $\Delta^{\HC}_{n-1}$. By 2., the sequence is $n$-complete.
\end{proof}

The following lemma is the analogue of \cite[Corollary 11 and Theorem 13]{KanoveiLyubetsky} for $n$-Jensen forcing notions. To prove their Theorem 13, Kanovei and Lyubetsky defined a forcing-like relation which allowed them to approximate the forcing relation for almost disjoint forcing when restricted to $\Sigma^0_n$ formulas (\cite[$\S5$]{KanoveiLyubetsky}). One can replicate this construction for generalised Jensen forcing; and with suitable checking of several properties which are proved in the almost disjoint case in \cite{KanoveiLyubetsky}, the proof of the following lemma goes through for forcing notions defined by $\Delta^{\HC}_{n-1}$, $n$-complete sequences in $\mathcal{M}_{\mathcal{J}}$.\footnote{For details, see \cite[Lemma 3.3.18]{wansner2023aspects} and its proof.}

\begin{lemma}[Kanovei-Lyubetsky Lemma]\label{theorem aabsoluteness}
For $n>2$, let $\langle (M_\xi, P_\xi) : \xi < \omega_1 \rangle$ be a $\Delta^{\HC}_{n-1}$, $n$-complete, and strictly $\preccurlyeq$-increasing $\mathcal{M}_{\mathcal{J}}$ sequence. Let $\pz^\ast$ be the $\omega_1$-product of $\bigcup_{\xi<\omega_1} P_\xi$, $G$ be $\pz^\ast$-generic over $L$, and $e \subseteq \omega_1$ be unbounded and in $L$. Then every $\Sigma^1_n$ formula with parameters in $L[G{\upharpoonright} e]$ is absolute between $L[G]$ and $L[G{\upharpoonright} e]$. 
\end{lemma}
\begin{proof}
The proof is similar to \cite[Corollary 11 and Theorem 13]{KanoveiLyubetsky}.
\end{proof}

Next, we provide a forcing-equivalence argument to show that Lemma \ref{theorem aabsoluteness} can be used to prove $n$-absoluteness for slices of an $\omega$-slice product.

\begin{proposition}\label{proposition Jensen slicing}
For $n>2$, let $\langle (M_\xi, P_\xi) : \xi < \omega_1 \rangle$ be a $\Delta^{\HC}_{n-1}$, $n$-complete, and strictly $\preccurlyeq$-increasing $\mathcal{M}_{\mathcal{J}}$ sequence. Let $\pz$ be the $\omega_1$-product of $\bigcup_{\xi<\omega_1} P_\xi$, $G$ be $\pz$-generic over $L$. Let $\pz$ be the $\omega$-slice product $\bigcup_{\xi<\omega_1} P_\xi$ of length $\omega_1$. Then $\pz$ is $n$-absolute for slices.
\end{proposition}

\begin{proof}
In $L$, there is a canonical bijection $b: \omega_1 \to \omega_1 \times \omega$. This $b$ induces an isomorphism between $\pz$ and $\pz^\ast$, the $\omega_1$-product from Lemma \ref{theorem aabsoluteness}. Now let $G^\ast$ be the $\pz^\ast$-generic filter over $L$ corresponding to $G$ induced by $b$, let $\mathcal{S} \subseteq \mathcal{P}( \omega_1 \times \omega)$ be an unbounded set of slices, and let $ s \in \mathcal{S}$ be a slice.  Note that $L[G{\upharpoonright}s] = L[G^\ast{\upharpoonright}{b^{-1}(s)}]$ and $b^{-1}(s)$ is unbounded in $\omega_1$. So, by Lemma \ref{theorem aabsoluteness}, $\pz$ is $n$-absolute for $\mathcal{S}$-slices. 
\end{proof}

\begin{definition}
Let $n \geq 2$. We call a Jensen-like forcing notion $\pz$ $n$-Jensen if 
\begin{enumerate}
    \item the set of $\pz$-generic reals over $L$ is $\Pi^1_n$ in every inner model of $\mathsf{ZFC}$ containing $L$, and
    \item the $\omega$-slice product of $\pz$ with finite support of length $\omega_1$ i $n$-absolute for slices.
\end{enumerate}
\end{definition}

Clearly, Jensen forcing is $2$-Jensen. Moreover, if $n > 2$ and $\langle (M_\xi, P_\xi) : \xi < \omega_1 \rangle$ is a $\Delta^{\HC}_{n-1}$, $n$-complete, and strictly $\preccurlyeq$-increasing $\mathcal{M}_{\mathcal{J}}$ sequence, then $\pz:= \bigcup_{\xi < \omega_1} P_\xi$ is $n$-Jensen. Hence, by Lemma \ref{lemma existence n complete sequence}, for every $n \geq 2$, there is an $n$-Jensen forcing notion in $L$. It is not known whether, up to forcing equivalence, there is exactly one $n$-Jensen forcing. We expect not.

Our final flourish is to prove Theorem \ref{theorem existence absolute for slices forcing notion}, i.e.\ to show that $n$-slicing forcing notions exist in $L$. Good candidates are $\omega$-slice products of $n$-Jensen forcing notions. However, we do not know whether the set $\{(\ell,x_G^{(\ell,k)}) : (\ell,k)\in \omega^2\}$ is $\Pi^1_n$. Recall that this was crucial in showing that the required amount of choice fails in Theorem \ref{thm:main}. The problem is that this does not follow from the fact that the set of all $\pz$-generic reals is $\Pi^1_n$ and Corollary \ref{corollary Jensen uniqueness}, because we cannot tell from which slice $\pz$-generic reals are. But we solve this issue by modifying $\omega$-slice products of $n$-Jensen forcing notions to use the first bit of generic reals to `track' which slice they are from. 

\begin{proof}[Proof of Theorem \ref{theorem existence absolute for slices forcing notion}]
Let $\pz^\ast$ be $n$-Jensen, let $\pz$ be the $\omega$-slice product of $\pz^\ast$ with finite support of length $\omega_1$, and let $\qz^k := \{T \in \pz^\ast : \langle k \rangle \subseteq \stem{T}\}$. Then $\qz^k$ is a perfect poset and if $D$ is dense in $\pz^\ast$, then $D \cap \qz^k$ is dense in $\qz^k$ for every $k \in \omega$. Hence, for every $k \in \omega$, every $\qz^k$-generic real is also $\pz^\ast$-generic.

Let $\qz$ be the $\omega$-slice product of $\langle \mathbb{Q}_\nu : \nu < \omega_1 \rangle$ with finite support, where
\begin{align*}
    \qz_{\nu} = 
    \begin{cases}
    \qz^{\nu+1} & \nu \in \omega\\
    \qz^0 & \text{otherwise,}
    \end{cases}
\end{align*}
let $H$ be a $\qz$-generic filter over $L$, and let $H':=\{p \in \pz : \exists q \in H (q \leq p)\}$. We show that $H'$ is $\pz$-generic over $L$. Let $D$ be dense in $\pz$. Then $D \cap \qz$ is dense in $\qz$ and so $H$ meets $D \cap \qz$. Therefore, $H'$ meets $D$. Thus, $H'$ is $\pz$-generic over $L$. As $H'$ can be constructed in $L[H]$ and vice versa, $L[H]=L[H']$. By Shoenfield absoluteness 
or  Proposition \ref{proposition Jensen slicing} respectively, $\qz$ is $n$-absolute for slices. 

It remains to show that $\{(\ell, x_H^{(\ell,k)}) : (\ell,k) \in \omega^2\}$ is $\Pi^1_n$ in $L[H]$. Let $B$ be the set of all $\pz^\ast$-generics in $L[H]$. By Corollary \ref{corollary Jensen uniqueness} 
, a real $y \in L[H]$ is $\pz^\ast$-generic iff there is some $(\nu , k) \in \omega_1 \times \omega$ such that $y = x^{(\nu,k)}_{H}$. Hence $B=\{x^{(\nu,k)}_{H'} : (\nu, k)\in \omega_1 \times \omega\}$. By definition of $H'$, $x_H^{(\nu,k)} = x_{H'}^{(\nu,k)}$ for every $(\nu,k) \in \omega_1 \times \omega$, so $B =\{x^{(\nu,k)}_{H} : (\nu, k)\in \omega_1 \times \omega\}$. By Corollary \ref{cor Jensen complexity set of generics}, $B$ is $\Pi^1_n$ in $L[H]$. Therefore,
$ \{(\ell, x_H^{(\ell,k)}) : (\ell,k) \in \omega^2\}= \{(\ell,x) : x(0)=\ell+1 \land x \in B\}$ is $\Pi^1_n$ in $L[H]$. Hence $\qz$ is $n$-slicing.
\end{proof}

We could use Kanovei and Lyubetsky's variant of almost disjoint forcing from \cite{KanoveiLyubetsky} to prove Corollary \ref{cor:1}, by adapting the proof of Theorem \ref{thm:main} to almost disjoint rather than Jensen forcing. This yields a model of $\ZF + \neg\mathsf{AC}_\omega(\mathbb{R};\mathbf{Ctbl})$. As the product of this variant is $n$-absolute for slices (see \cite[Theorem 9]{KanoveiLyubetsky}), $\mathsf{DC}(\R;\Pi^1_n)$ also holds in this model. A compactness argument then yields Corollary \ref{cor:1}. However, our proof of Theorem \ref{thm:main} does not work for almost disjoint forcing, as the generics are not unique: we can swap finitely many bits in a generic to generate another generic (as in  \cite[Lemma 9vi)]{KanoveiLyubetsky}). In which case, we do not know whether $\mathsf{AC}(\R;\Pi^1_{n+1})$ holds.

\section{Open questions}

So far, we have only considered the hierarchy of uniform choice principles for `$\mathbf{\Pi}$' projective point classes. This suffices, as the `$\mathbf{\Pi}$'- and `$\mathbf{\Sigma}$'-type uniform choice principle hierarchies are equivalent in the following sense:

\begin{proposition}[Kanovei, {\cite[pg. 6]{kanovei1979descriptive}}]\label{prop unif choice principles sigma pi correspond}
Under $\ZF$, $\AC_\omega(\R;\mathrm{unif}\mathbf{\Pi}^1_n)$ is equivalent to $\AC_\omega(\R;\mathrm{unif}\mathbf{\Sigma}^1_{n+1})$.
\end{proposition}

By contrast, for the non-uniform choice principles, we know only that $\AC_\omega(\R;\mathbf{\Pi}^1_{n+1})$ implies $\AC_\omega(\R;\mathbf{\Pi}^1_n)$ and $\AC_\omega(\R;\mathbf{\Sigma}^1_n)$. 

\begin{ques}
Is $\AC_\omega(\R;\mathbf{\Pi}^1_n)$ $\ZF$-equivalent to $\AC_\omega(\R;\mathbf{\Sigma}^1_{n+1})$? 
\end{ques}

Such an equivalence does not follow from the uniform equivalence (i.e.\ Proposition \ref{prop unif choice principles sigma pi correspond}), as Corollary \ref{cor:3} shows that $\AC_\omega(\R;\Gamma)$ is not $\ZF$-equivalent to $\AC_\omega(\R;\mathrm{unif}\Gamma)$.

As well as the equivalences of non-uniform descriptive choice principles, we can ask about separating non-uniform descriptive choice principles. By Corollary \ref{cor:3}, we can separate uniform choice for all projective pointclasses from the minimal non-uniform descriptive  choice principle beyond $\ZF$ (namely $\AC_\omega(\mathbf{F}_\sigma)$). It is not known whether the non-uniform projective choice principles can be separated from one another:

\begin{ques}\label{question nonuniform level by level}
Is there a model $M$ which satisfies $\ZF+\AC_\omega(\R;\mathbf{\Pi}^1_n)$ which violates $\AC_\omega(\R;\Pi^1_{n+1})$?
\end{ques}

Lastly, we defined our principles for any pointclass, so we can ask about Borel choice principles. In $\ZFC$, several definitions are equivalent to the ordinary definition of the Borel hierarchy  (i.e.\ where $\mathbf{\Sigma}^0_{\alpha} :=\{\bigcup_{n \in \omega} A_n : A_n \in \bigcup_{\beta<\alpha}\mathbf{\Pi}^0_\beta\}$). These need not coincide under $\ZF$, hence their corresponding Borel choice principles may be distinct.  For example, the `closure' definition (where  $\widehat{\mathbf{\Sigma}}^0_{\alpha}$ is the closure of $\bigcup_{\beta<\alpha}\widehat{\mathbf{\Pi}}^0_\beta$ under countable unions) is possibly distinct in $\ZF$, e.g.\ possibly $\widehat{\mathbf{\Sigma}}^0_2 \supsetneq \mathbf{\Sigma}^0_2$ (see \cite{miller2008long, miller2011dedekind} for examples). The codeable-Borel hierarchy is a further such definition (see \cite[\S562]{fremlin2000measure}). We can  ask whether the various Borel choice principles can be separated, e.g.:

\begin{ques}\label{question borel}
Is there a model $M$ which satisfies $\ZF$ and $\AC_\omega(\R;\mathbf{\Sigma}^0_\alpha)$ which violates $\AC_\omega(\R;\mathbf{\Sigma}^0_{\alpha+1})$? Is there a model which satisfies $\ZF$ and $ \AC_\omega(\R;\mathrm{unif}\mathbf{\Sigma}^0_\alpha)$ which violates $\AC_\omega(\R;\mathrm{unif}\mathbf{\Sigma}^0_{\alpha+1})$?
\end{ques}

\bibliographystyle{alpha}
\bibliography{2023 paper}

\begin{thebibliography}{FGK19}

\bibitem[And20]{andretta2001notes}
Alessandro Andretta.
\newblock Notes on descriptive set theory, 2020.
\newblock In preparation.

\bibitem[Coh66]{cohen2008set}
Paul~J Cohen.
\newblock {\em Set theory and the continuum hypothesis}.
\newblock W. A. Benjamin, Inc., 1966.

\bibitem[FGK19]{friedman2018model}
Sy-David Friedman, Victoria Gitman, and Vladimir Kanovei.
\newblock A model of second-order arithmetic satisfying {AC} but not {DC}.
\newblock {\em Journal of Mathematical Logic}, 19(01):1850013, 2019.

\bibitem[FL63]{feferman1963independence}
Solomon Feferman and Azriel L\'{e}vy.
\newblock Independence results in set theory by {C}ohen’s method {II}.
\newblock {\em Notices of the American Mathematical Society}, 10:593, 1963.

\bibitem[Fre08]{fremlin2000measure}
David~Heaver Fremlin.
\newblock {\em Measure theory}, volume~5.
\newblock Torres Fremlin, 2008.

\bibitem[GHJ16]{GHJZFCwopowerset}
Victoria Gitman, Joel~David Hamkins, and Thomas~A. Johnstone.
\newblock What is the theory {ZFC} without power set?
\newblock {\em Mathematical Logic Quarterly}, 62(4-5):391--406, 2016.

\bibitem[Jec03]{jech2003set}
Thomas Jech.
\newblock {\em Set theory}, volume~14 of {\em Springer Monographs in
  Mathematics}.
\newblock Springer, 2003.

\bibitem[Jen70]{jensen}
Ronald Jensen.
\newblock Definable sets of minimal degree.
\newblock In {\em Mathematical logic and foundations of set theory ({P}roc.
  {I}nternat. {C}olloq., {J}erusalem, 1968)}, pages 122--128, 1970.

\bibitem[Kan79]{kanovei1979descriptive}
Vladimir Kanovei.
\newblock On descriptive forms of the countable axiom of choice.
\newblock {\em Investigations on nonclassical logics and set theory, Work
  Collect}, pages 3--136, 1979.

\bibitem[KL17]{MR3691702}
Vladimir Kanovei and Vassily Lyubetsky.
\newblock A definable countable set containing no definable elements.
\newblock {\em Matematicheskie Zametki}, 102(3):369--382, 2017.

\bibitem[KL20]{KanoveiLyubetsky}
Vladimir Kanovei and Vassily Lyubetsky.
\newblock Models of set theory in which nonconstructible reals first appear at
  a given projective level.
\newblock {\em Mathematics}, 8(6), 2020.

\bibitem[Mil08]{miller2008long}
Arnold~W Miller.
\newblock Long {B}orel hierarchies.
\newblock {\em Mathematical Logic Quarterly}, 54(3):307--322, 2008.

\bibitem[Mil11]{miller2011dedekind}
Arnold~W Miller.
\newblock A {D}edekind finite {B}orel set.
\newblock {\em Archive for Mathematical Logic}, 50(1):1--17, 2011.

\bibitem[Wan23]{wansner2023aspects}
Lucas Wansner.
\newblock {\em Aspects of Forcing in Descriptive Set Theory and Computability
  Theory}.
\newblock PhD thesis, Universit{\"a}t Hamburg, 2023.

\bibitem[Won23]{wontner}
Ned Wontner.
\newblock {\em Views from a peak: Generalisations and Descriptive Set Theory}.
\newblock PhD thesis, University of Amsterdam, 2023.

\end{thebibliography}

\end{document}